\documentclass[preprint,12pt]{elsarticle}
\usepackage{amssymb,amsmath,amsfonts,amsthm,mathtools,epsfig}
\newtheorem{thm}{Theorem}
\newtheorem{theorem}[thm]{Theorem}
\newtheorem{lemma}[thm]{Lemma}

\journal{Discrete Mathematics}

\begin{document}

\begin{frontmatter}

\title{Analogues of Bermond-Bollob\' as Conjecture for Cages Yield Expander Families}


\author{Leonard Chidiebere Eze and Robert Jajcay}

\affiliation{organization={Department of Algebra and Geometry, Faculty of Mathematics, Physics and Informatics, Comenius University},%Department and Organization
            addressline={Mlynska Dolina}, 
            postcode={84248}, 
            city={Bratislava},
            country={Slovakia}}

%\author{Robert Jajcay}
%\author[inst1,inst2]{Author Three}

%\affiliation[inst2]{organization={Department Two},%Department and %Organization
%            addressline={Address Two}, 
 %%           city={City Two},
 %           postcode={22222}, 
 %           state={State Two},
 %           country={Country Two}}

\begin{abstract}
%Both cages and expander graphs are the focus of intense research interests, but have rarely been studied together. 
This paper presents a possible link between Cages and Expander Graphs by introducing three interconnected variants of the Bermond and Bollobás Conjecture, originally formulated in 1981 within the context of the Degree/Diameter Problem. We adapt these conjectures to cages, with the most robust variant posed as follows: 
    \textit{Does there exist a constant $c$ such that for every pair of parameters $k,g$ there exists a $k$-regular graph of girth $g$ and order not exceeding $ M(k,g) + c $?}; where $M(k,g)$ denotes the value of the so-called Moore bound for cages.
    We show that a positive answer to any of the three variants of the Bermond and Bollob\' as Conjecture for cages considered in our paper would yield expander graphs (expander families); thereby establishing a connection between Cages and Expander Graphs.
\end{abstract}


\end{frontmatter}

\section{Introduction}
Cages and expander graphs are two seemingly distinct classes of graphs that have been the focus of research in theoretical computer science and mathematics for the past five decades (see, e.g., the surveys \cite{ExooR1, Linial}). Nevertheless, both classes share common use in building robust, high-performance communication networks, and play a crucial role in a host of other applications, from complexity theory to coding theory.

The {\em Cage Problem} is the problem of finding {\em $(k,g)$-cages}, which are the smallest $k$-regular graphs of girth $g$, for each $k \geq 3 $ and $g \geq 3$.
The problem has been widely studied since the pioneering works of Erd\H{o}s and Sachs \cite{Erds} and of Hoffman and Singleton \cite{Hoffman} in the 1960's. It is well-known that for any given pair of integers $k\geq 3$ and $ g\geq 3$, there exist infinitely many $(k,g)$-graphs ($k$-regular graphs of girth $g$) \cite{Erds}, but the problem of determining {\em the orders of the smallest $(k,g)$-regular graphs}, denoted by $n(k,g)$, is wide open for the majority of parameter pairs $(k,g)$ \cite{ExooR1}. 

The second problem considered in our paper, the {\em Expander Graph Problem},  asks for constructions of infinite families of finite regular graphs that are sparse and highly connected (with the precise definition involving either spectral or combinatorial properties of the graphs). Determination of cages (or even devising constructions of small $(k,g)$-graphs) and constructions of expander graphs are challenging combinatorial and computational problems. Although many individual graphs, such as complete graphs and some $(k,g)$-cages, are known to have good expansion properties, but constructing explicit infinite families of expander graphs has proven to be a challenge. In this paper, we aim to construct an expander family from the perspective of the Cage Problem by extending the Bermond-Bollob\'as Conjecture to cages and assuming positive answers to any of the three variants presented herein.

In trying to construct expander families, various families of graphs have been studied, including extremal graphs of given maximal degree and diameter \cite{Filipovski} and regular graphs with fixed second largest eigenvalue \cite{Yang}.
   The Cage Problem is also related to the well-known problem in
   Extremal Graph Theory -- the {\em Degree/Diameter Problem} -- in which one asks for maximal orders of graphs of given maximum degree and diameter \cite{siran}. A quick comparison of the Cage and Degree/Diameter Problems shows that both problems are related through the so-called Moore bound (see definition below; note that the original Moore bound was stated
only for the case of odd girth). For the Degree/Diameter Problem, we denote {\em the order of the largest graphs of maximum degree $\Delta$ and diameter $D$}
by $n(\Delta,D)$. This
order is bounded from above by the Moore bound $M(\Delta,D)$ expressed as

$$M(\Delta,D) = 1 + \Delta + \Delta(\Delta-1) + \dots + \Delta(\Delta -1)^{D-1}.$$

Any $\Delta$-regular graph that achieves this bound must be of girth $ 2D + 1$. Therefore, the Moore bound $M(\Delta,D)$ is a lower bound on the order of $(k,g)$-graphs with $k=\Delta$ and $g=2D+1$. Consequently, for $ k=\Delta$ and $g=2D+1$, $M(\Delta,D) = M(k,g) \leq n(k,g)$, for all $k \geq 3$ and $g \geq 3 $, where $M(k,g)$ denotes {\em the Moore bound for the order of a $k$-regular graph of girth
$g$}. Considering both the odd and even girths $g$, the order $|V(\Gamma)|$ of any $(k,g)$-graph satisfies the following inequalities:
\vskip -5mm

$$
|V(\Gamma)| \geq n(k,g)\ge M(k,g) = $$
$$ \left\{
\begin{array}{ll}
	1+\sum_{i=0}^{(g-3)/2}k(k-1)^{i}=\dfrac {k(k-1)^{(g-1)/2}-2}{k-2}, & \mbox{$g$ odd,}\\
	2\sum_{i=0}^{(g-2)/2}(k-1)^{i}=\dfrac {2(k-1)^{g/2}-2}{k-2}, & \mbox{ $g$ even. }
\end{array} \right.
$$

\noindent Graphs that achieve the Moore bound are called {\em Moore graphs}. Clearly, all Moore graphs are cages. If the order of a $k$-regular graph $\Gamma$ of girth $g$ exceeds the corresponding Moore bound, we call the difference between its order and the Moore bound its {\em excess $e(\Gamma)$}.

In 1981, Bermond and Bollob\'as \cite{Bollobas} raised the following question: 
\vskip -2mm

\begin{center}
\textit{Is it true that for each integer $c>0$ there exist $\Delta>2$ and $D\geq2$ such that the order of the largest graph of maximum degree $\Delta$ and diameter $D$ is at most $M(\Delta, D)-c$?} 
\end{center}
\vskip -2mm

More than forty years later, we still do not know the answer to this question. 
While studying this question (possibly trying to prove a positive answer), 
Jajcay and Filipovski \cite{Filipovski} linked the possibility of a negative answer 
to the existence of a class of expander graphs called Ramanujan graphs by showing that a negative answer
to the question of Bermond and Bollob\'as 
would yield for any fixed $\Delta$ and all sufficiently large even $D$ a family of expander graphs called Ramanujan graphs defined in the next paragraph via the use of spectral properties of graphs.

Let $\Gamma$ be a $k$-regular graph of order $n$. It is known that eigenvalues of the adjacency matrix of $\Gamma$ are of the form $\lambda_{0} \geq \lambda_{1} \geq \lambda_{2} \geq \dots \geq \lambda_{n-1}$ with $\lambda_{0}=k$ and $ |\lambda_{i}| \leq k$ for $i=1, 2, \dots, n-1$. If the absolute value of the second largest eigenvalue of $\Gamma$ is $\lambda $ (i.e., $\lambda = \max \{|\lambda_{1}|,|\lambda_{n-1}|\} $), then we say that $\Gamma$ has a {\em $\lambda$-expansion}, and we refer to $\Gamma$ as an {\em $(n, k, \lambda)$-graph}. 
%The bigger the $\lambda$, the better the expansion of $\Gamma$. 
An $(n, k, \lambda)$-graph $\Gamma$ is called a {\em Ramanujan graph} if $\lambda \leq 2 \sqrt{k-1}$ \cite{Alon, Lubotzky}. It is well known that Ramanujan graphs are {\em expander graphs} \cite{Lubotzky}. %That is to say, Ramanujan graphs have good expansion properties. 
We also know that it is not difficult to construct small Ramanujan graphs. For example, complete graphs $K_{k+1}$ have been shown to be Ramanujan graphs. However, the main challenge is to determine whether there exist infinite families of $k$-regular Ramanujan graphs for all degrees $k$. This question has been partly answered for every prime $k-1$. In 1988, Lubotzky, Phillips, and Sarnak \cite{Lubotzky} proved that if $k-1$ is a prime, there are infinite families of $k$-regular Ramanujan graphs. In a later development, Friedman \cite{Friedman} showed that for every $ \epsilon > 0 $ almost all $k$-regular graphs satisfy $\lambda(\Gamma) \leq 2\sqrt{k-1} +\epsilon$. In 2013, Marcus, Spielman, and Srivastava \cite{Marcus} showed that bipartite Ramanujan expanders exist for all degrees $ k \geq 3 $. 

Another approach to determining expander graphs is by applying the concept of the {\em Cheeger constant}. The Cheeger constant of a graph measures the connectivity of the graph, and we describe it as follows. Let $\Gamma$ be a finite graph, and $\emptyset \neq S \subset V(\Gamma)$. Define $\sigma(S) = \{(u,v) \in E(\Gamma): u \in S , v \in V(\Gamma)\backslash S\}$. Then,  the \textit{Cheeger constant $h(\Gamma)$} of the graph $\Gamma$ is defined as
\begin{equation}
	h(\Gamma) = \min \left\{ \dfrac{|\sigma(S)|}{|S|}: S\subset V,~~ 0 < |S| \leq \frac{|V|}{2} \right\}. 	
\end{equation}
%	\end{dfn}  

\noindent
Interestingly, the second largest eigenvalue and the Cheeger constant of a graph are related by the following {\em Cheeger inequality}.
		
\begin{theorem}[\cite{Alon}]
		Let $\Gamma$ be a $k$-regular connected graph, then
			\begin{equation} \label{cheeger}
			 \frac{k - \lambda}{2} \leq h(\Gamma) \leq \sqrt{2k(k-\lambda)},
			 \end{equation}
			where $\lambda$ is the second largest eigenvalue of $\Gamma$.
\end{theorem}

\noindent
An infinite family of $k$-regular graphs $ \{ \Gamma_n \}_{n=1}^{\infty} $ is said
to be an {\em expander family} if there exists a positive number $c$ such that 
$ h(\Gamma_n)>c $, for all $ n \in {\Bbb N} $.

\section{Variants of Bermond-Bollob\' as Conjecture for Cages}
    A positive answer to the Bermond and Bollob\' as question in the Degree/Diameter
Problem would not necessarily rule out the existence of interesting infinite families of 
    graphs of maximal degree $k$ and increasing diameter $d$ or of families of fixed diameter 
    $ d > 2 $ and increasing maximum degree $k$ whose orders are meaningfully close to the Moore bound.

For this reason, we pose three related variants of the Bermond and
Bollob\' as question in the case of cages; with the first one being the closest analogue of
the original Bermond and Bollob\' as question as stated in the context of the Degree/Diameter Problem. 

\begin{description}	
\item {\bf BB1} {\em Does there exist a constant $c$ such that for every pair of parameters $(k,g)$ there exists a $(k,g)$-graph of order not exceeding $ M(k,g)+c $?}
\item {\bf BB2} {\em Does there exist for every $ k \geq 3 $ a constant $c_k$ such that for every $g \geq 3$ there exists a $(k,g)$-graph of order not exceeding $ M(k,g)+c_k $?}
\item{\bf BB3} {\em Does there exist for every $ k \geq 3 $ a constant $c'_k$ such that there exist infinitely many $g \geq 3$ with the property that there exists a $(k,g)$-graph of order not exceeding $M(k,g)+c'_k$?}
\end{description}
Obviously, a positive answer to {\bf BB1} would imply a positive answer to {\bf BB2}
and {\bf BB3} while a negative answer to {\bf BB1} might still allow for a positive answer to {\bf BB2} and {\bf BB3}. Furthermore, a positive answer to {\bf BB2} would imply a positive answer to {\bf BB3} while a negative answer to {\bf BB2} might still allow for a positive answer to {\bf BB3}. 

Notably, questions {\bf BB2} and {\bf BB3} could also 
be stated in the form where one would fix the girth and allow the degree to vary. These kind of variants would also be interesting, but we chose not to 
state them explicitly, as we do not address them in this paper. It is known that infinite families
of $k$-regular graphs of girths $ 6 $, $ 8 $, and $ 12 $ of orders exactly equal to the
Moore bound exist for degrees $k=p+1 $, $ p $ a prime \cite{ExooR1}. However, nothing similar is known for odd or larger even girths.

We also wish to point out that in the more specialized case of vertex-transitive cages, the first two variants of this question have been answered in negative,
	and the excess of vertex-transitive $(k,g)$-cages can be arbitrarily large.
	This result is due to Biggs:
	\begin{theorem}[\cite{B}]
		\label{biggs-old}
		For each odd integer $k\geq3$, there is an infinite sequence of values of $g$ such that the excess $e(\Gamma)$ of any vertex-transitive graph $\Gamma$ of degree $k$ and girth $g$ satisfies $e(\Gamma)>\dfrac{g}{k}.$
	\end{theorem}
	
	In \cite{Filipovski}, the authors show that Biggs' result in \cite{B} holds not only for infinitely many
	$g$'s, but, in fact, holds for {\em almost all} $g$'s, for any fixed $k \geq 4$. Specifically,
	they show that for any given excess $e(\Gamma)$ and degree $k \geq 4$, the set of $g$'s for which
	$ n_{vt}(k,g) - M(k,g) < e(\Gamma) $, is of asymptotic density $0$ (when compared to the set of all girths
	$g \geq 3$), where {\em $ n_{vt}(k,g)$ is the order of the smallest vertex-transitive $(k,g)$-graphs}. The main technique used in the paper relies on counting cycles in vertex-transitive graphs whose orders are close to the Moore bound.
	
	In this paper, we consider hypothetical $(k,g)$-graphs of orders smaller than $ M(k,g) + c_k$, 
	where $ c_k$ is a positive constant.
	% (we do not know whether there exist infinitely many graphs with these properties). 
	We prove that for every fixed $ k \geq 3 $ and every $ c_k$ there exists a positive constant $ N(k,\epsilon) $ such that the Cheeger constant of any $(k,g)$-graph of order less than or equal to $ M(k,g) + c_k$  is greater than or equal to $ N(k,\epsilon) $. 
This implies that if either of {\bf BB1}, {\bf BB2}, or {\bf BB3} were true, then for each $k \geq 3 $ there would exist an expander family of $ (k,g) $-graphs with increasing $g$.

\section{Odd Girth}
In what follows, we repeatedly refer to a complete $k$-regular tree of depth $s$ which is a tree rooted in a vertex of degree $k$, all the vertices of this tree that are of distance smaller than $s$ from the root are of degree $k$, and all the leaves are of distance $s$ from the root. The order of this tree is equal to $ M(k,2s+1) $, and for obvious reasons we shall call it the {\em Moore tree of degree $k$ and depth} $s$, and denote it by $ {\mathcal T}_{k,s} $.

Let $ \Gamma $ be a $(k,g)$-graph. % of order $ \alpha M(k,g) $, $ \alpha \geq 1 $. 
The following lemma is an easy consequence of the well-known fact that every vertex $u$ of a $(k,g)$-graph $ \Gamma $ is a root of a Moore tree of depth $s$ contained in $ \Gamma $. We denote this Moore tree by $ {\mathcal T}_{k,s}^u $.

\begin{lemma}\label{easy-counting}
	Let $g=2s+1$, $ g \geq 3$, and let $ \Gamma $ be a $(k,g)$-graph. 
	For every vertex $ v \in V(\Gamma) $ and every $ 0 \leq s' \leq s $
	there exist exactly $ M(k,2s'+1) $ vertices $ u \in V(\Gamma) $ such that the Moore tree $ {\mathcal T}_{k,s'}^u $ 
	contains $v$. 
\end{lemma}
\begin{proof} A vertex $ v \in V(\Gamma) $ is contained in $ {\mathcal T}_{k,s'}^u $, for some $ u \in V(\Gamma) $, if and only if, $d_{\Gamma}(v,u) \leq s' $. Since these are exactly the vertices contained in $ {\mathcal T}_{k,s'}^v $, there are $ M(k,2s'+1) $ vertices $ u \in V(\Gamma) $ with this property, and the result follows.
\end{proof}

The above argument yields the following.
\begin{lemma}\label{fraction1}
	Let $g=2s+1 \geq 3$, $ \Gamma $ be a $(k,g)$-graph, %of order $ \alpha M(k,g) $, $ \alpha \geq 1 $, 
	and $ \beta = \frac{M(k,g-2)}{|V(\Gamma)|} $. If $ \emptyset \neq S \subseteq V(\Gamma) $, 
	then there exists a 
	vertex $ u \in V(\Gamma) $ whose Moore tree $ {\mathcal T}_{k,s-1}^u $ contains at least $ \beta |S| $ vertices
	from $S$, i.e., $ | V({\mathcal T}_{k,s-1}^u) \cap S | \geq  \beta |S| $.
\end{lemma}
\begin{proof}
	The proof proceeds via contradiction. Suppose the lemma is not true and for every $ u \in V(\Gamma) $ we have
	the inequality $ | V({\mathcal T}_{k,s-1}^u) \cap S | <  \beta |S| $. Then, 
	\[ \sum_{u \in V(\Gamma)}  | V({\mathcal T}_{k,s-1}^u) \cap S | < \beta \cdot |V(\Gamma)| \cdot |S| . \]
	On the other hand, Lemma~\ref{easy-counting} yields that each vertex $ v \in S $ appears in exactly 
	$ M(k,g-2) $ trees $ {\mathcal T}_{k,s-1}^u $, $ u \in V(\Gamma) $. Therefore, 
	\[ \sum_{u \in V(\Gamma)}  | V({\mathcal T}_{k,s-1}^u) \cap S | = |S| \cdot M(k,g-2) . \]
	Combining the two results %we derived for $ \sum_{u \in V(\Gamma)}  | V({\mathcal T}_{k,s-1}^u) \cap S | $ 
 yields
	the inequality
	\[ |S| \cdot M(k,g-2) < \beta \cdot  |V(\Gamma)| \cdot |S| \]
	which implies the desired contradiction 
	\[ M(k,g-2) <  \frac{M(k,g-2)}{|V(\Gamma)|} \cdot |V(\Gamma)|. \]
\end{proof}

Throughout the rest of this section, we will use a generalized concept of a graph that allows semi-edges; edges incident with one vertex only. Such generalized graphs are called {\em multipoles}. The degree of a vertex $u$ in a multipole is the number of edges and semi-edges incident to $u$. Furthermore, the concept of an induced subgraph in a multipole is also slightly different from the usual usage of induced subgraphs of simple graphs. Namely, for a non-empty subset $S$ of vertices of a finite multipole $ \Gamma $ (which may or may not contain semi-edges), 
the {\em induced
	multipole of} $ \Gamma $ {\em determined by the subset} $S$, $ \Gamma^M(S) $, 
is the multipole with vertex set $S$ in which each vertex $ u \in S $ is
adjacent to all the vertices $v \in S $ to which $u$ was adjacent in $\Gamma$ (the `usual' induced edges) while $u$ is also
incident to a semi-edge for each edge $ \{ u,v \} $ of $\Gamma$ with $ v \not \in S $ and $u$ is incident to each semi-edge of $\Gamma$
incident with $u$. Thus, unlike the case of an induced subgraph of a graph, all the vertices in an induced multipole $ \Gamma^M(S) $
are of the same degree as they were in $\Gamma$.

The main result of this section now follows from Lemma~\ref{fraction1}.
\begin{theorem}\label{main1}
	Let  $ k \geq 3 $, and suppose that there exists a constant $c_k$ and an infinite increasing sequence of odd girths
	$ \{ g_i \}_{i \in {\Bbb N}} $ such that for each $g_i$ in the sequence there exists a $(k,g_i)$-graph $\Gamma_{k,g_i}$ of order not exceeding 
	$ M(k,g_i)+c_k $. Then, for 
	every $ \epsilon \geq 0 $, there exists $ N_{\epsilon} $ such that $ \{ \Gamma_{k,g_i} \}_{i \geq N_{\epsilon}} $ is an expander family with the Cheeger constant of each of the graphs $\Gamma_{k,g_i}$, $ i \geq N_{\epsilon} $, greater than or equal to 
	\[ \frac{1}{k-1} - \epsilon . \]
\end{theorem}
\begin{proof}
	Let $ g_i = 2s_i+1 $, and $ \Gamma_{k,g_i} $ be a $ (k,g_i) $-graph of order not exceeding $M(k,g_i)+c_k$, ( where $ M(k,g_i)+c_k = \frac{k(k-1)^{s_i}-2}{k-2} + c_k $). Let $S$ be a non-empty subset of $ V(\Gamma_{k,g_i}) $ with size at most 
	$ \frac{|V(\Gamma_{k,g_i})|}{2} $, and let $ \beta_{k,g_i} = \frac{M(k,g_i-2)}{|V(\Gamma_{k,g_i})|} \geq \frac{k(k-1)^{s_i-1}-2}{k(k-1)^{s_i}+c_k(k-2)} $. Lemma~\ref{fraction1} asserts the existence of a vertex $ u \in V(\Gamma_{k,g_i}) $ such that the set $ S_u = 
	V({\mathcal T}_{k,s_i-1}^u) \cap S $ is of cardinality at least $ \beta_{k,g_i} |S| $. %= \frac{k(k-1)^{s_i-1}-2}{k(k-1)^{s_i}+c_k(k-2)} |S| $.
	Suppose $ \tilde{S_u} $ denote the set $ S - S_u $, then the lower bound on the order of $|S_u|$ yields an upper bound on $ |\tilde{S_u}| $,
	namely, $ |\tilde{S_u}| = |S| - |S_u| \leq |S| - \beta_{k,g_i} |S| = (1-\beta_{k,g_i}) |S| $. Consider the multipole $ \Gamma^M_{k,g_i}(S_u) $. Since 
	$ \Gamma_{k,g_i} $ is $k$-regular, the induced multipole $ \Gamma^M_{k,g_i}(S_u) $ is also $k$-regular. Moreover, $ S_u $ is a subset of
	$ V({\mathcal T}_{k,s_i-1}^u) $. Therefore $ \Gamma^M_{k,g_i}(S_u) = ({\mathcal T}_{k,s_i}^u)^M(S_u) $, where $ {\mathcal T}_{k,s_i}^u $
	is a tree (and thus contains no cycles). It follows that $ \Gamma^M_{k,g_i}(S_u) $ contains no cycles. Hence, the number of edges 
	with both end-vertices contained in $ \Gamma^M_{k,g_i}(S_u) $ is at most $ |V( \Gamma^M_{k,g_i}(S_u))| -1 = |S_u| -1 $, and all other (semi-)edges 
	of $ \Gamma^M_{k,g_i}(S_u) $ are in fact true semi-edges; each incident with exactly one vertex in $S_u$. Furthermore, all semi-edges
	of $ \Gamma^M_{k,g_i}(S_u) $ stem from edges of $ \Gamma_{k,g_i} $ incident with a vertex in $S_u$ and a vertex in $ V(\Gamma_{k,g_i})
	- S_u $; which come in two kinds. First, there are edges of $ \Gamma_{k,g_i} $ incident to a vertex in $S_u$ and $ \tilde{S_u} $. Recalling again that $ {\mathcal T}_{k,s_i -1}^u $ is a tree, it follows that no two distinct edges in $ \Gamma_{k,g_i} $ incident to a vertex in $S_u$ and a vertex in $ \tilde{S_u} $ can be incident to the same vertex in $ \tilde{S_u} $ as this would form a cycle of length less than $ g_i=2s_i+1 $. Thus, the number of edges of $ \Gamma_{k,g_i} $ incident to a vertex in $S_u$ and $ \tilde{S_u} $ is bounded from above by $ |\tilde{S_u}| \leq (1-\beta_{k,g_i}) |S| $. This bound also yields a lower bound on the 
	number of edges of $ \Gamma_{k,g_i} $ of the second kind, i.e., edges incident to a vertex in $S_u$ and $ V(\Gamma_{k,g_i}) - S $:  
	\[ k |S_u| - |S_u| +1 - (1-\beta_{k,g_i}) |S| = |S_u| (k-1) +1 - (1-\beta_{k,g_i}) |S|; \]
	the total number of edges and semi-edges in the multipole $ \Gamma_{k,g_i}(S_u) $ minus an upper bound on the number of edges
	in $ \Gamma_{k,g_i}(S_u) $ minus an upper bound on the number of edges between $S_u$ and $ \tilde{S_u} $. If we recall that the number of edges $ \Gamma_{k,g_i} $ incident to a vertex in $S_u$ and $ V(\Gamma_{k,g_i}) - S $ is a lower bound on the number $ |\sigma(S)| $ of edges incident to a vertex in $S$ and $ V(\Gamma_{k,g_i}) - S $, then for every 
	non-empty subset $S$ of $ V(\Gamma_{k,g_i}) $, %of order at most $ \frac{|V(\Gamma_{k,g_i})|}{2} $, 
		\begin{equation} \label{eq2}
		\begin{split}
			|\sigma(S)| &\geq |S_u| (k -1) +1 - (1-\beta_{k,g_i}) |S| \\ &\geq 
			\beta_{k,g_i} |S|(k-1) +1 - (1-\beta_{k,g_i}) |S| \\ &= |S| (\beta_{k,g_i}k - 1)+1 .
		\end{split}
	\end{equation}
	This yields a lower bound on the Cheeger constant of $ \Gamma_{k,g_i} $:
	
	\begin{equation*} \label{eq3}
		\begin{split}
	 \min_{0 < |S| \leq \frac{|V(\Gamma_{k,g_i})|}{2}} \{ \frac{|\sigma(S)|}{|S|} \}  &\geq 
	 \min_{0 < |S| \leq \frac{|V(\Gamma_{k,g_i})|}{2}} \{ \beta_{k,g_i}k - 1 + \frac{1}{|S|} \} \\ &\geq 
	 \beta_{k,g_i}k - 1 + \frac{2}{|V(\Gamma_{k,g_i})|} \\ & \geq 
	 \frac{k(k-1)^{s_i-1}-2}{k(k-1)^{s_i}-2+c_k(k-2)}k-1 + \frac{2}{\frac{k(k-1)^{s_i}-2}{k-2} + c_k} \\ & =
	 \frac{k(k-1)^{s_i-1}-2}{k(k-1)^{s_i}-2+c_k(k-2)}k-1 + \frac{2(k-2)}{k(k-1)^{s_i}-2 + c_k(k-2)} \\ & = 
	 \frac{k^2(k-1)^{s_i-1}-2k-k(k-1)^{s_i} +2-c_k(k-2)+2k -4}{k(k-1)^{s_i}-2+c_k(k-2)} \\ &=
	  \frac{k^2(k-1)^{s_i-1}-2+c_k(k-2)}{k(k-1)^{s_i}-2+c_k(k-2)} - \frac{k(k-1)^{s_i} + 2c_k(k-2)}{k(k-1)^{s_i}-2 + c_k(k-2)} .
   %\\ 
   %& \simeq 
%\lim_{s_i \to \infty} \left(  \frac{k^2(k-1)^{s_i-1}+c_k(k-2)}{k(k-1)^{s_i}-2+c_k(k-2)} - \frac{k(k-1)^{s_i} + c_k(k-2)}{k(k-1)^{s_i}-2 + c_k(k-2)} \right)  \\ & = \frac{1}{k-1}. 
		\end{split}
\end{equation*}
	Since, 
\begin{eqnarray*}
    \lim_{s_i \to \infty} \left(  \frac{k^2(k-1)^{s_i-1}+c_k(k-2)}{k(k-1)^{s_i}-2+c_k(k-2)} - \frac{k(k-1)^{s_i} + 2c_k(k-2)}{k(k-1)^{s_i}-2 + c_k(k-2)} \right) = \\
    \lim_{s_i \to \infty} \frac{k^2(k-1)^{s_i -1} -2+ c_k(k-2)}{k(k-1)^{s_i}-2 + c_k(k-2)} - \lim_{s_i \to \infty} \frac{k(k-1)^{s_i} + 2c_k(k-2)}{k(k-1)^{s_i}-2 + c_k(k-2)} = \\
    \frac{k}{k-1} - 1 = \frac{1}{k-1},   
\end{eqnarray*} 

 %$ \lim_{s_i \to \infty} \frac{k^2(k-1)^{s_i -1} -2+ c_k(k-2)}{k(k-1)^{s_i}-2 + c_k(k-2)} = \frac{k}{k-1} $ and $ \lim_{s_i \to \infty} \frac{k(k-1)^{s_i} + c_k(k-2)}{k(k-1)^{s_i}-2 + c_k(k-2)} = 1 $, 
	the result follows.
\end{proof}

\section{Even Girth}
Having resolved the case of odd girth, the next natural case to study is that of even girth. %In this section, we consider the possibility of an even girth. 
Fortunately, the same counting argument used for odd girth can also be applied here. However, we emphasize that for even girth $g=2s$, instead of using a Moore tree $  {\mathcal T}_{k,s}^u $ (rooted at a vertex $u$), we use a Moore tree $  {\mathcal T}_{k,s}^e $ associated with two vertices $u$ and $v$ that are connected by an edge $e$, the {\em root edge}. The reason is that a Moore tree for a $(k,g)$-graph of even girth $g=2s$ consists of two trees of depth $s-1$ rooted at two vertices of degree $k-1$, incident to the root edge $e$. It is easy to see that the order of $  {\mathcal T}_{k,s}^e $ corresponds to the Moore
bound for cages of even girth given in the Introduction.
%is the The vertices of these trees other than the root vertices
%and their leaves are all of degree $k$. This gives a $2 \times (k-1)$ vertices at the leave layer. Further, adding $k-1$ incident edges to each vertex in the leave layer gives $2 \times (k-1)^2$. Observe that repeating this process gives the Moore bound as stated in the introductory section above. Moreover, it is easy to observe that for $g=2s$, the distance from a leave vertex to the root vertex is $s-1$. Thus, the depth of the Moore tree is $s-1$.

\begin{lemma}\label{easy-counting12}
	Let $g=2s \geq 4$, and let $ \Gamma $ be a $(k,g)$-graph. 
	For every vertex $ v \in V(\Gamma)$ and every $ 0 \leq s' \leq s-1 $,
	there exist exactly $ M(k,2s'+1) - 1$ edges $ e \in E(\Gamma) $ such that the Moore tree $ {\mathcal T}_{k,s'}^e $ 
	contains $v$. 
 
\end{lemma}
\begin{proof} A vertex $ v \in V(\Gamma) $ is contained in $ {\mathcal T}_{k,s'}^e $, for some $ e \in E(\Gamma) $, if and only if the distance of $v$ to at least one of the endpoints 
of $e$ is less than or equal to $s'-1 $.
%d$\min\{d_{\Gamma}(\omega,u), d_{\Gamma}(\omega, v)\} \leq s' $. 
Since the number of such edges is equal to the number of edges contained in the tree $ {\mathcal T}_{k,s'}^v $ (rooted at a {\em vertex} $v$), there are $ M(k,2s'+1) -1 $ edges $ e \in E(\Gamma) $ with this property (recall that $ {\mathcal T}_{k,s'}^v $ is
a tree), and the result follows. Here, it is important to emphasize that the tree 
$ {\mathcal T}_{k,s'}^v $ is a vertex-rooted tree (just like the trees in the previous section) which is not necessarily a subtree of $\Gamma$, 
since $\Gamma$ is of girth $g=
2s$ and the depth of $ {\mathcal T}_{k,s'}^v $ is $ s' \leq s $. The extra $-1$ in 
the expression $ M(k,2s'+1) -1 $ is due to the fact that {\em we are counting edges and not vertices} in the tree $ {\mathcal T}_{k,s'}^v $.
\end{proof}

\begin{lemma}\label{fraction12}
	Let $g=2s \geq 4$, let $ \Gamma $ be a $(k,g)$-graph, %of order $ \alpha M(k,g) $, $ \alpha \geq 1 $, 
		and let $ \beta = \frac{M(k,2s-1)-1}{|E(\Gamma)|} = \frac{2M(k,2s-1)-2}{k|V(\Gamma)|}$. Suppose $ \emptyset \neq S \subseteq V(\Gamma) $. 
	Then, there exists an edge $ e \in E(\Gamma) $ whose Moore tree $ {\mathcal T}_{k,s-1}^e $ contains at least $ \beta |S| $ vertices
	from $S$, i.e., $ | V({\mathcal T}_{k,s-1}^e) \cap S | \geq  \beta |S| $.
\end{lemma}
\begin{proof}
	Suppose for contradiction that the claim in Lemma \ref{fraction12} is not true, that is, for every $ e \in E(\Gamma) $
	the inequality $ | V({\mathcal T}_{k,s-1}^e) \cap S | <  \beta |S| $. Then, 
	\[ \sum_{e \in E(\Gamma)}  | V({\mathcal T}_{k,s-1}^e) \cap S | < \beta \cdot |E(\Gamma)| \cdot |S| . \]
	Applying Lemma~\ref{easy-counting12}, we observe that for each vertex $ v \in S $ there are exactly
	$ M(k,2s-1) - 1$ trees $ {\mathcal T}_{k,s-1}^e $, $ e \in E(\Gamma) $, that contain $v$. Thus,
	\[ \sum_{e \in E(\Gamma)}  | V({\mathcal T}_{k,s-1}^e) \cap S | = |S| \cdot (M(k,2s-1)-1) . \]
	Combining the two results %we derived for $ \sum_{u \in V(\Gamma)}  | V({\mathcal T}_{k,s-1}^e) \cap S | $ 
	yields the inequality
	\[ |S| \cdot (M(k,2s-1)-1) < \beta \cdot  |E(\Gamma)| \cdot |S| \]
	which implies the desired contradiction 
	\[ M(k,2s-1)-1 <  \frac{M(k,2s-1)-1}{|E(\Gamma)|} \cdot |E(\Gamma)|. \]
\end{proof}

In the sequel, we apply Lemma~\ref{fraction12} to prove the main result of this section.
\begin{theorem}\label{main12}
	Let  $ k \geq 3 $, and suppose that there exists a constant $c_k$ and an infinite increasing sequence of even girths
	$ \{ g_i \}_{i \in {\Bbb N}} $ such that for each $g_i$ in the sequence there exists a $(k,g_i)$-graph $\Gamma_{k,g_i}$ of order not exceeding 
	$ M(k,g_i)+c_k $. Then, for 
	every $ \epsilon \geq 0 $, there exists $ N_{\epsilon} $ such that $ \{ \Gamma_{k,g_i} \}_{i \geq N_{\epsilon}} $ is an expander family with the Cheeger constant of each of the graphs $\Gamma_{k,g_i}$, $ i \geq N_{\epsilon} $, greater than or equal to 
	\[ \frac{1}{k-1} - \epsilon . \]
\end{theorem}
\begin{proof}
	Let $ g_i = 2s_i $, and $ \Gamma_{k,g_i} $ be a $ (k,g_i) $-graph such that $$|\Gamma_{k,g_i}| \leq M(k,g_i)+c_k = \frac{2(k-1)^{s_i}-2}{k-2} + c_k. $$ 
	%of order not exceeding 	$ M(k,g_i)+c_k = \frac{2(k-1)^{s_i}-2}{k-2} + c_k $. 
Let $\varnothing \neq S \subseteq V(\Gamma_{k,g_i}) $ be of cardinality at most 
	$ \frac{|V(\Gamma_{k,g_i})|}{2} $, and let $\beta_{k,g_i}$ be the ratio of the size of ${\mathcal T}_{k,s_i-1}^e $ to the size of $ \Gamma_{k,g_i},$ i.e.,
\begin{eqnarray*}	 \beta_{k,g_i} = \frac{M(k,2s_i-1)-1}{|E(\Gamma_{k,2s_i-1})|} = \frac{2M(k,2s_i-1)-2}{k|V(\Gamma_{k,2s_i-1})|} 
	\geq  \frac{2M(k,2s_i-1)-2}{k(M(k,2s_i)+c_k)} = \\
\frac{\frac{2k(k-1)^{s_i-1}-4-2(k-2)}{k-2}}{\frac{2k(k-1)^{s_i}-2k+c_kk(k-2)}{k-2}} =
\frac{2k(k-1)^{s_i-1}-2k}{2k(k-1)^{s_i}-2k+c_kk(k-2)} .
%	 \frac{2(k-1)^{s_i-1}-2}{2(k-1)^{s_i}+c_k(k-2)} . 
\end{eqnarray*} 
By Lemma~\ref{fraction12}, there exists a vertex $ u \in V(\Gamma_{k,g_i}) $ with the set $ S_u = 
	V({\mathcal T}_{k,s_i-1}^e) \cap S $ and $ |S_u| \geq \beta_{k,g_i} |S|. $ %= \frac{k(k-1)^{s_i-1}-2}{k(k-1)^{s_i}+c_k(k-2)} |S| $.
	Now, if we let $ \tilde{S_u} $ denote the set $ S - S_u $, then it follows from the proof of Theorem \ref{main1} that $ |\tilde{S_u}| \leq (1-\beta_{k,g_i}) |S| $. Consider the multipole $ \Gamma^M_{k,g_i}(S_u) $. Since 
	$ \Gamma_{k,g_i} $ is $k$-regular, the induced multipole $ \Gamma^M_{k,g_i}(S_u) $ is also $k$-regular. Furthermore, $ S_u $ is a subset of
	$ V({\mathcal T}_{k,s_i-1}^e) $, and therefore $ \Gamma^M_{k,g_i}(S_u) = ({\mathcal T}_{k,s_i}^e)^M(S_u) $, where $ {\mathcal T}_{k,s_i}^e $
	is a tree with root edge $e$. % (and thus contains no cycles). 
 It follows that $ \Gamma^M_{k,g_i}(S_u) $ contains no cycles, and hence the number of edges 
	with both end-vertices contained in $ \Gamma^M_{k,g_i}(S_u) $ is at most $ |V( \Gamma^M_{k,g_i}(S_u))| -1 = |S_u| -1 $, and all other (semi-)edges 
	of $ \Gamma^M_{k,g_i}(S_u) $ are in fact true semi-edges; each incident with exactly one vertex in $S_u$. Moreover, all semi-edges
	of $ \Gamma^M_{k,g_i}(S_u) $ stem from edges of $ \Gamma_{k,g_i} $ incident with a vertex in $S_u$ and $ V(\Gamma_{k,g_i})
	- S_u $; which also come in two kinds. First, there are edges of $ \Gamma_{k,g_i} $ incident to a vertex in $S_u$ and 
	$ \tilde{S_u} $. Recalling again that $ {\mathcal T}_{k,s_i}^e $ is a tree, it follows that no two distinct edges in $ \Gamma_{k,g_i} $ incident to  a vertex in $S_u$ and a vertex in $ \tilde{S_u} $ can be incident to the same vertex in $ \tilde{S_u} $ as this would cause the 
	existence of a cycle of length less than $ g_i=2s_i $. Thus, the number of edges of $ \Gamma_{k,g_i} $ incident with a vertex in $S_u$ and $ \tilde{S_u} $ is bounded from above by $ |\tilde{S_u}| \leq (1-\beta_{k,g_i}) |S| $. This yields a lower bound on the 
	number of edges of $ \Gamma_{k,g_i} $ of the second kind, i.e., edges incident to a vertex in $S_u$ and 
	$ V(\Gamma_{k,g_i}) - S $:  
	\[ k |S_u| - |S_u| +1 - (1-\beta_{k,g_i}) |S| = |S_u| (k-1) +1 - (1-\beta_{k,g_i}) |S|; \]
	the total number of edges and semi-edges in the multipole $ \Gamma_{k,g_i}(S_u) $ minus an upper bound on the number of edges
	in $ \Gamma_{k,g_i}(S_u) $ minus an upper bound on the number of edges between $S_u$ and $ \tilde{S_u} $. By observing that the number of edges of $ \Gamma_{k,g_i} $ incident to a vertex in $S_u$ and $ V(\Gamma_{k,g_i}) - S $ is a lower bound on the number $ |\sigma(S)| $ of edges incident to a vertex in $S$ and $ V(\Gamma_{k,g_i}) - S $, we obtained that for every 
	non-empty subset $S$ of $ V(\Gamma_{k,g_i}) $, %of order at most $ \frac{|V(\Gamma_{k,g_i})|}{2} $ 
	\begin{equation} \label{eq4}
		\begin{split}
			 |\sigma(S)| &\geq |S_u| (k -1) +1 - (1-\beta_{k,g_i}) |S| \\ &\geq 
			 \beta_{k,g_i} |S|(k-1) +1 - (1-\beta_{k,g_i}) |S| \\ &= |S| (\beta_{k,g_i}k - 1)+1 .
		\end{split}
	\end{equation}
	
	This yields a lower bound on the Cheeger constant of $ \Gamma_{k,g_i} $:
		\begin{equation*} \label{eq4}
		\begin{split}
	 \min_{0 < |S| \leq \frac{|V(\Gamma_{k,g_i})|}{2}} \{ \frac{|\sigma(S)|}{|S|} \} & \geq
	\min_{0 < |S| \leq \frac{|V(\Gamma_{k,g_i})|}{2}} \{ \beta_{k,g_i}k - 1 + \frac{1}{|S|} \} \\ & \geq 
	 \beta_{k,g_i}k - 1 + \frac{2}{|V(\Gamma_{k,g_i})|} \\ & \geq 
	\frac{2k(k-1)^{s_i-1}-2k}{2k(k-1)^{s_i}-2k+c_k k(k-2)}k-1+  \frac{2}{\frac{2(k-1)^{s_i}-2}{k-2} + c_k} \\ & = 
	 \frac{2k(k-1)^{s_i-1}-2k}{2(k-1)^{s_i}-2+c_k(k-2)}-1 + \frac{2(k-2)}{2(k-1)^{s_i}-2 + c_k(k-2)} \\ & = 
	\frac{2k(k-1)^{s_i-1}-2k-2(k-1)^{s_i} +2-c_k(k-2)+2k -4}{2(k-1)^{s_i}-2+c_k(k-2)} \\ & =
	  \frac{2k(k-1)^{s_i-1}-2+c_k(k-2)}{2(k-1)^{s_i}-2+c_k(k-2)} - \frac{2(k-1)^{s_i} + 2c_k(k-2)}{2(k-1)^{s_i}-2 + c_k(k-2)} . 
   %\\ & \simeq 
	%\lim_{s_i \to \infty} \left(  \frac{2k(k-1)^{s_i-1}-2+c_k(k-2)}{2(k-1)^{s_i}-2+c_k(k-2)} -  \frac{2(k-1)^{s_i} + 2c_k(k-2)}{2(k-1)^{s_i}-2 + c_k(k-2)}\right) \\& = \frac{1}{k-1}. 
  \end{split}
\end{equation*}
	Since, 
\begin{eqnarray*}
	\lim_{s_i \to \infty} \left(  \frac{2k(k-1)^{s_i-1}-2+c_k(k-2)}{2(k-1)^{s_i}-2+c_k(k-2)} -  \frac{2(k-1)^{s_i} + 2c_k(k-2)}{2(k-1)^{s_i}-2 + c_k(k-2)}\right) = \\   \lim_{s_i \to \infty} \frac{2k(k-1)^{s_i} -2+ c_k(k-2)}{2(k-1)^{s_i}-2 + c_k(k-2)} - \lim_{s_i \to \infty} \frac{2(k-1)^{s_i} + 2c_k(k-2)}{2(k-1)^{s_i}-2 + c_k(k-2)} = \\ \frac{k}{k-1} - 1 = \frac{1}{k-1} , 
 \end{eqnarray*}
 
 the result follows.
\end{proof}

	\section{Concluding Remarks}
Combining the results of the previous two sections provides us with the main
theorem of this paper:
\begin{theorem}\label{main3}
Let  $ k \geq 3 $, and let  $c_k$ be a positive constant for which there exists an infinite increasing sequence of girths
$ \{ g_i \}_{i \in {\Bbb N}} $ and an infinite family of $(k,g_i)$-graphs $ \{ \Gamma_{k,g_i}\}_{i \in {\Bbb N}}$ of orders not exceeding $ M(k,g_i)+c_k $. Then, for 
every $ \epsilon \geq 0 $, there exists an infinite subsequence $ \{ g_{i_j} \}_{j \in {\Bbb N}} $
of the sequence $ \{ g_i \}_{i \in {\Bbb N}} $ such that $ \{ \Gamma_{k,g_{i_j}} \}_{j \in {\Bbb N}} $ is an expander family with the Cheeger constant greater than or equal to $ \frac{1}{k-1} - \epsilon $.
\end{theorem}

\begin{proof}
An infinite increasing sequence of girths $ \{ g_i \}_{i \in {\Bbb N}} $ necessarily 
contains an infinite subsequence of odd or an infinite subsequence of even girths.
Applying the appropriate theorem of Theorems~\ref{main1} and \ref{main12} yields
the desired results.
\end{proof}

Interestingly, a careful analysis of the proofs of Theorems~\ref{main1} and \ref{main12} suggests the possibility of  further strengthening the 
results in a way that would assure the existence of an expander subfamily of 
$k$-regular graphs even in a sequence of $(k,g_i)$-graphs whose orders
do not exceed $ M(k,g_i)+f(g_i) $ for a sufficiently slowly increasing 
function $f$. Since we wanted to keep our investigation in line with the
original question of Bermond and Bollob\'as, we did not pursue this 
direction. Instead, 
before concluding our article, we wish to point out that it is not even known whether 
there exists a $ k \geq 3 $ and a corresponding constant $C_k$ such that there exist infinitely 
many $g_i \geq 3$ with the property that there exists a $(k,g_i)$-graph of order not exceeding $C_kM(k,g_i)$; a constant {\em multiple} of the Moore bound. In view of this, 
one might be tempted to ask whether a similar result to that of Theorem~\ref{main3}
might exist for constant multiples of $ M(k,g) $, i.e., whether any family of $(k,g_i)$-graphs whose orders do not exceed $C_kM(k,g_i)$, for any given constant $C_k$, must contain a subfamily of expanders. 

The answer to such question stated in this most general form is negative. It is 
not true that 
{\em any} family of $(k,g_i)$-graphs whose orders do not exceed $C_kM(k,g_i)$, for {\em any} given constant $C_k$, necessarily contains a family of expanders. The argument for this claim goes as follows.

Let $k \geq 3 $, and suppose the existence of an infinite family of $(k,g_i)$-graphs $\Gamma_{k,g_i}$, $ i \in {\Bbb N} $, of orders not exceeding $ \alpha_k M(k,g_i) $ for some fixed $ \alpha_k > 1 $ 
(if no such family and no such $\alpha_k$ exist, the above question is obviously moot). 
Let $\overline{\Gamma}_{k,g_i}$ be a family of graphs constructed from the graphs 
$\Gamma_{k,g_i}$ by taking two disjoint copies 
$\Gamma_{k,g_i}^1$ and $\Gamma_{k,g_i}^2$ of $\Gamma_{k,g_i}$, selecting the same edge $\{ u^1,v^1 \}$ and $\{ u^2,v^2 \}$ in each of the 
copies, removing the edges $\{ u^1,v^1 \}$ and $\{ u^2,v^2 \}$ and replacing them
with edges $\{ u^1,v^2 \}$ and $\{ u^2,v^1 \}$. It is not hard to see that the 
resulting family
$\overline{\Gamma}_{k,g_i}$, $ i \in {\Bbb N} $, is a family of $ (k,g_i) $-graphs 
of orders not exceeding $ 2 \alpha_k M(k,g_i) $. Moreover, it is also not hard to
see that this new family does not contain a subfamily
of expanders: The number of edges between the two complementary subsets of vertices 
belonging to $ \Gamma_{k,g_i}^1$ and $\Gamma_{k,g_i}^2$ of $\overline{\Gamma}_{k,g_i}$
is always $2$ regardless of the order of $\Gamma_{k,g_i} $. Thus, even if the 
family $\Gamma_{k,g_i}$, $ i \in {\Bbb N} $, contained a family of expanders, the 
family $\overline{\Gamma}_{k,g_i}$, $ i \in {\Bbb N} $, would not.

Since the above `trick' of connecting two copies of a vertex-transitive graph does not
necessarily produce a vertex-transitive graph, it still might be the case
that any infinite family of {\em vertex-transitive} $(k,g_i)$-graphs $ \Gamma_{k,g_i} $ of 
orders not exceeding $ \alpha_k M(k,g_i) $, for any constant $ \alpha_k > 1 $, necessarily contains a subfamily of expanders. However, a similarly simple trick preserving the vertex-transitivity of the 
constructed graphs (possibly creating graphs of orders which are a constant 
multiple of the orders of the original graphs where the constant is larger 
than $2$) may also resolve this question in negative. Therefore, we leave this
question for further investigation.

Finally, in view of the results obtained in \cite{Filipovski} where the authors have shown that 
the existence of an infinite family of graphs of fixed degree and increasing diameters of orders 
differing from the Moore bound by at most a constant would necessarily lead to a family of 
Ramanujan graphs, one should ask whether it might be possible to prove a similar result in case
of cages. If such a result were possible, it would probably have to be proven by radically different
techniques. Invoking the Cheeger inequality (\ref{cheeger}) and using the Cheeger constant 
determined in Theorem~\ref{main3}, one only obtaines the inequalities
\[ \frac{k-\lambda(\Gamma_{k,g_i})}{2} \leq \frac{1}{k} \leq \sqrt{2k(k-\lambda(\Gamma_{k,g_i}))} \]
implying the inequalities 
\[ k - \frac{2}{k} \leq \lambda(\Gamma_{k,g_i}) \leq k - \frac{1}{2k^3} , \] 
which is quite far from being able to prove that $  \lambda(\Gamma_{k,g_i})  \leq 2 \sqrt{k-1} $.
The only way to remedy this approach would depend on finding a better (larger) Cheeger constant
for the considered families, which we were not able to do, and which might not even exist.

\section{Acknowledgments} \vskip -2mm
Both authors acknowledge the support from VEGA 1/0437/23 and the second author is also supported by APVV-19-0308.

\section{Declaration of competing interest}
 The authors have no financial or personal relationships that could influence this work. 
 
 \section{Data availability}
 
 No data was used for the research described in the article.

%\lipsum[2]
%\lipsum[3]


\begin{thebibliography}{5}

\bibitem{Alon}	
		N. Alon, Eigenvalues and expanders, Combinatorica, {\bf 6(2)} (1986) 83--96.


      \bibitem{Bollobas}
		J. C. Bermond and B. Bollob\'as, The diameter of graphs: A survey, Congressus Numerantium, {\bf 32} (1981) 3--27.	

  \bibitem{B} 
		N. L. Biggs,
		Excess in vertex-transitive graphs, 
		Bull. London Math. Society, {\bf 14} (1982) 52--54.
		
		\bibitem{Biggs}
		N. L. Biggs,
		Constructions for cubic graphs of large girth,
		Electron. J. Combin., {\bf 5} (1998).

\bibitem{Erds}
		P. Erd\H{o}s and H. Sachs,
		Regul\" are Graphen gegebener Taillenweite mit minimaler Knotenzahl.	Wiss. Z. Uni. Halle (Math. Nat.), {\bf 12} (1963) 251--257.

               \bibitem{ExooR1} 
		G. Exoo and R. Jajcay,
		Dynamic cage survey,
		Electron. J. Comb., Dynamic Survey, {\bf 16} (2013).

		\bibitem{Filipovski}
		S. Filipovski and R. Jajcay, A connection between a question of Bermond and Bollob\'as and Ramanujan graphs, Acta Appl. Math., {\bf 175(1)} (2021) 1--10.
					
            \bibitem{Friedman}
		J. Friedman,
		A proof of Alon's second eigenvalue conjecture	and related problems, Mem. Amer. Math. Soc., {\bf 195(910)} (2008). 
  
           \bibitem{Hoffman}
		A. J. Hoffman and R. R. Singleton,
		On Moore graphs with diameters 2 and 3,
		IBM	J. Res. Develop., {\bf 4} (1960) 497--504.

		

				\bibitem{Linial}
		S. Hoory, N. Linial, and A. Wigderson, 
		Expander graphs and their applications, 
		Bull. Am. Math. Soc., New Ser., {\bf 43(4)} (2006) 439--561.
		
		
            \bibitem{Lubotzky}
		A. Lubotzky, R. Phillips, and P. Sarnak, Ramanujan graphs, Combinatorica, {\bf 8(3)} (1988) 261--277.
		
		\bibitem{Marcus}
		A. W. Marcus, D. A. Spielman, and N. Srivastava,
		Interlacing families I: Bipartite Ramanujan graphs of all degrees, 
        Ann. Math. (2) 182, No. 1, 307-325 (2015).
%  FOCS, (2013) 529--537.

\bibitem{siran}
		M. Miller and J. \v Sir\'a\v n,
	    Moore graphs and beyond: A survey of the degree/diameter problem, 
	    Electron. J. Combin., Dynamic Survey, {\bf 14} (2005).
     
		\bibitem{Yang}
	    J. Y. Yang and J. H. Koolen, On the order of regular graphs with fixed second largest eigenvalue, Linear Algebra and its Applications, {\bf 610} (2021) 29--39.
     
		
	      
		
		
						

	     
			
				
	\end{thebibliography}
\end{document}